\documentclass[11pt,reqno]{amsart}
\usepackage{amsthm, amssymb, latexsym, multicol, verbatim, enumerate, accents, float, hyperref, amsmath, amscd, youngtab, enumitem, cleveref}
\usepackage[usenames]{color}
\usepackage[utf8]{inputenc}

\advance\textwidth by 1.2in  \advance\oddsidemargin by -.6in \advance\evensidemargin by -.6in

\theoremstyle{plain}
\newtheorem{thm}{Theorem}[section]
\newtheorem*{thm*}{Theorem}
\newtheorem{prop}[thm]{Proposition}
\newtheorem{lem}[thm]{Lemma}
\newtheorem{conj}[thm]{Conjecture}

\theoremstyle{definition}
\newtheorem{defn}[thm]{Definition}

\newtheorem{rem}[thm]{Remark}

\newcommand{\lie}[1]{\mathfrak{#1}}   \newcommand\bc{\mathbb C} \newcommand\bn{\mathbb N}

\newcommand{\lm}{\operatorname{lm}}

\newcommand{\lc}{\operatorname{lc}
}

\newcommand{\lexp}{\operatorname{lexp}}
\newcommand{\Mon}{\operatorname{Mon}}

\renewcommand\l\lambda
\newcommand\m\mu
\renewcommand\d\partial

\newcounter{cnt}
 \makeatletter
\def\mydggeometry{\makeatletter\dg@YGRID=1\dg@XGRID=20\unitlength=0.003pt\makeatother}
\makeatother \theoremstyle{remark}

\numberwithin{equation}{section}
\makeatletter
\def\section{\def\@secnumfont{\mdseries}\@startsection{section}{1}%
  \z@{.7\linespacing\@plus\linespacing}{.5\linespacing}%
  {\normalfont\scshape\centering}}
\def\subsection{\def\@secnumfont{\bfseries}\@startsection{subsection}{2}%
  {\parindent}{.5\linespacing\@plus.7\linespacing}{-.5em}%
  {\normalfont\bfseries}}
\makeatother



\begin{document}
\title[Gr\"obner bases for fusion products]{Gr\"obner bases for fusion products}

\author{Johannes Flake}
\address{RWTH Aachen University, Pontdriesch 10-16, 52062 Aachen}
\email{flake@art.rwth-aachen.de}
\author{Ghislain Fourier}
\email{fourier@art.rwth-aachen.de}
\author{Viktor Levandovskyy}
\email{levandov@math.rwth-aachen.de}

\subjclass[2010]{}
\begin{abstract}
We provide a new approach towards the analysis of the fusion products defined by B.~Feigin and S.~Loktev in the representation theory of (truncated) current Lie algebras. We understand the fusion product as a degeneration using Gr\"obner theory of non-commutative algebras and outline a strategy on how to prove a conjecture about the defining relations for the fusion product of two evaluation modules. We conclude with following this strategy for $\mathfrak{sl}_2(\mathbb{C}[t]) $ and hence provide yet another proof for the conjecture in this case.
\end{abstract}
\maketitle
\thispagestyle{empty}
\section{Introduction}
In the framework of finite-dimensional modules for current algebras, B.~Feigin and S.~Loktev introduced the fusion product of evaluation modules. 
This is on one hand the ordinary tensor product of two simple finite-dimensional modules of a semisimple finite-dimensional complex Lie algebra $\lie g$, for our purposes $\lie{sl}_n(\mathbb{C})$, and at the same time a graded module for $\lie g \otimes \bc[t]$, the current algebra. 
If $V(\lambda)$ and $V(\mu)$ are two simple highest-weight modules, then one obtains by construction graded Littlewood-Richardson coefficients $c_{\lambda, \mu}^{\tau}(q)$, a main motivation for the introduction of fusion products. 
In the following years, fusion products in general played their role in the construction of local Weyl modules and Demazure modules for $\lie g \otimes \mathbb{C}[t]$. 
Despite their relevance in the representation theory of current algebras in the past twenty years, important properties are still not proved. 
For example, fusion products are cyclic modules for $\lie g \otimes \bc[t]$ but their defining ideals are not known. 
A conjecture by E.~Feigin (Conjecture~\ref{conj-fusion}), claiming that the obvious relations are actually the defining relations, remains open in general. \\
Local Weyl modules have been defined and studied for $\lie g \otimes A$ for any unital, finitely generated, commutative algebra $A$ (and even beyond this case), but character formulas are known only for $\bc, \bc[t]$ and $\bc[t^{\pm}]$. 
Even for $A = \bc[t]/(t^2)$, the character and dimension of the local Weyl modules is conjectured only. The conjecture about the defining relations of the fusion product would provide this information (Theorem~\ref{thm-main}).\\  
For a fixed $m$, the Clebsch-Gordan formula for representations of $\lie{sl}_2(\bc)$ gives rise to a poset $\{(k,m-k) \mid 0 \leq k \leq m\}$ with $$(k, m-k) \geq (\ell, m- \ell) :\Leftrightarrow |m- 2k| \leq | m- 2 \ell|,$$ saying that there is an injective map of $\lie{sl}_2(\mathbb{C})$-modules $V(\ell) \otimes V(m - \ell) \longrightarrow V(k) \otimes V(m - k)$. 
Generalizing to $\lie{sl}_n(\bc)$, one has an induced partial order, using all positive roots, on the set $\mathcal{P}_\lambda = \{ (\lambda_1, \lambda_2) \mid \lambda = \lambda_1 + \lambda_2, \lambda_i \text{ dominant, integral weights}  \}$. 
A conjecture, formulated for the first time in \cite{DP07} (for variants see also \cite{FFLP05}, \cite{LPP07}), states that 
\[
(\lambda_1, \lambda - \lambda_1) \geq (\mu, \lambda - \mu) \Rightarrow V(\mu) \otimes V(\lambda - \mu) \hookrightarrow V(\lambda_1) \otimes V(\lambda - \lambda_1)
\]
or equivalently, denoting by $s_\lambda$ the Schur function or character of $V(\lambda)$, $s_{\lambda - \lambda_1}  s_{\lambda_1} - s_{\lambda- \mu} s_{\mu}$ is a non-negative sum of Schur functions. 
This conjecture is known as the Schur positivity conjecture and again, there are only partial results proved so far. 
Again, the proof of the defining relations for the fusion product would imply the conjecture on Schur positivity (Theorem~\ref{thm-main}).\\

The impact and relevance of the conjecture defining relations should be clear by now but unfortunately a proof is known in a few cases only. 
The $\lie{sl}_2(\bc)$-case follows from the Clebsch-Gordan formula, further cases are treated for example in \cite{Fou15}. Moreover, providing the proof for the case $\lambda \gg \mu$ has initiated the framework of PBW degenerations, see \cite{FFL11} and \cite{CIFFFR17}. \\
In this paper, we provide a new approach to attack the problem. 
The fusion product is defined as the associated graded module with respect to the natural filtration (induced by the degree function on $\bc[t]$) and hence it is natural to describe the problem in terms of Gr\"obner degenerations. 
Due to the context, we have to deal with Gr\"obner theory for non-commutative, infinite-dimensional algebras, hence the existence of an appropriate finite  basis is not clear.
The fusion product is defined using two evaluation parameters, but one can simplify the general conjecture to just one parameter, say $a \in \bc$, then the defining ideal $I_a(\lambda, \mu)$ for the tensor product of the evaluation modules is known. 
We fix a 
monomial well-ordering on 
$U(\lie g \otimes \bc[t]/I_a)$
compatible with the degree ordering on $\bc[t]$. 
Suppose that there exists a Gr\"obner basis for  $I_a(\lambda, \mu)$ whose leading terms do not contain the parameter $a$, then the ideal of leading terms (defined as in \Cref{rem-leading-ideal}) equals the ideal of 
leading terms of the ideal proposed by E.~Feigin (Conjecture~\ref{conj-flat}). 
We summarize our construction in our main theorem
\begin{thm*}
The existence of an appropriate finite Gr\"obner basis implies the conjecture on the defining relations for the fusion product.
\end{thm*}
We are left with finding 
such a Gr\"obner basis for $I_a(\lambda, \mu)$, which  would be almost hopeless in general, but 
is much more accessible while  working 
in the context of $G$-algebras. 
 The first step in the general proof is
\begin{thm*} There exists an appropriate finite Gr\"obner basis for the fusion product of two evaluation modules for $\lie{sl}_2(\bc[t])$.
\end{thm*}

The combinatorics of this Gr\"obner basis in the $\lie{sl}_2(\bc[t])$-case is non-trivial and we expect that $\lie{sl}_n(\bc[t])$ will be difficult as well. 

The paper is organized as follows: In Section~\ref{sec-setup} we recall the fusion product and the setup from representation theory, and translate the problem to degeneration theory. Section~\ref{sec-grobner} provides the input from Gr\"obner theory of $G$-algebras and the proof that our approach is valid. 
Section~\ref{sec-conj} explains the impact on the various conjectures on Schur positivity and local Weyl modules, while in Section~\ref{sec-sl2}, we provide a Gr\"obner basis for the case of $\lie{sl}_2(\bc)$.

\bigskip

\textbf{Acknowledgments} J.~Flake and V.~Levandovskyy are supported through the DFG-CRC 195 ``Symbolic Tools in Mathematics and their Applications``.

\section{Representation Theory - Setup}\label{sec-setup}
Let $\lie g$ be a finite-dimensional simple complex Lie algebra and $\lie g \otimes \bc[t]$ its current (Lie) algebra with Lie bracket
\[
[ x \otimes p, y \otimes q ] = [x,y]_{\lie g} \otimes pq.
\]
We consider $\lie g$ as embedded into $\lie g \otimes \bc[t]$ by the map $x \mapsto x \otimes 1$.\\
We fix a triangular decomposition $\lie g = \lie n^+ \oplus \lie h \oplus \lie n^-$, denote the set of positive roots by $R^+$ and for each $\alpha \in R^+$, we fix a (non-zero) root vector $e_\alpha$ while fixing a root vector $f_\alpha$ for $-\alpha$. 
Further denote $h_\alpha = [ e_\alpha, f_\alpha]$.\\
The finite-dimensional simple $\lie g$-modules are indexed by their highest (dominant integral) weight $\lambda$, the set of dominant integral weights is denoted by $P^+$, and the fundamental weights are denoted by $\omega_i \in P^+$. For $\lambda \in P^+$, let $V(\lambda)$ be the corresponding simple module and let $v_\lambda$ be any non-zero vector in the highest weight space. 
By denoting the universal enveloping algebra as $U(\lie g)$, we have 
$$V(\lambda) = U(\lie g).v_\lambda = U(\lie n^-).v_\lambda.
$$

Let $V$ be a $\lie g$-module and $a \in \bc$, then
\[
(x \otimes p).v := p(a)x.v
\]
defines a $\lie g \otimes \bc[t]$-module structure on $V$ and we denote this evaluation module by $V_a$. Let $a_1, \ldots, a_s \in \bc$ be pairwise distinct and $\lambda_1, \ldots, \lambda_s \in P^+$, then
$$
V(\lambda_1, \ldots, \lambda_s, a_1, \ldots, a_s):= V(\lambda_1)_{a_1} \otimes \cdots \otimes V(\lambda_s)_{a_s}
$$ 
is a simple $\lie g \otimes \bc[t]$-module, and any finite-dimensional simple $\lie g \otimes \bc[t]$-module is of this form. It is important to notice here that
$$
V(\lambda_1, \ldots, \lambda_s, a_1, \ldots, a_s) \cong_{\lie g} V(\lambda_1) \otimes  \ldots \otimes V(\lambda_s),
$$
hence the $\lie g$-module structure is independent of $a_1,\dots,a_s$.
\\

B.~Feigin and S.~Loktev introduced the fusion product for modules of the current algebra \cite{FL99}, whose definition we recall here. These tensor products of evaluation modules are not graded (with respect to the natural grading on $\bc[t]$), but only filtered, so they constructed the associated graded module. One main motivation is a natural construction of graded Littlewood-Richardson coefficients. \\
The algebra $U(\lie g \otimes \bc[t])$ is naturally graded by the degree in $t$, i.e. 
$$
\deg \left((x_1 \otimes p_1) \cdots (x_s\otimes p_s)\right) = \sum_{i=1}^s \deg p_i ,
$$ 
so for $r\in\mathbb{N}_0$
\[
U(\lie g \otimes \bc[t])_{r} = \{ z \in U(\lie g \otimes \bc[t]) \mid \deg z \leq r\}
\]
defines a filtration on $U(\lie g \otimes \bc[t])$. Each filtered component is naturally a $\lie g$-module, and the associated graded algebra is again isomorphic to $U(\lie g \otimes \bc[t])$.

For each simple $U(\lie g \otimes \bc[t])$-module $V(\lambda_1, \ldots, \lambda_s, a_1, \ldots, a_s)$, we fix a highest weight vector $v$ of weight $\lambda_1 + \ldots + \lambda_s$. Then
\[
U(\lie g \otimes \bc[t])_{\leq r} \,. v  \subset V(\lambda_1, \ldots, \lambda_s, a_1, \ldots, a_s)
\]
defines a filtration on $V(\lambda_1, \ldots, \lambda_s, a_1, \ldots, a_s)$. The associated graded space is then a module for $U(\lie g \otimes \bc[t])$, with each graded component being a $\lie g$-module. This is called the fusion product and is denoted by
\[
V(\lambda_1)_{a_1} \ast \ldots \ast V(\lambda_s)_{a_s}.
\]
Again, we remark that 
$$
V(\lambda_1)_{a_1} \ast \ldots \ast V(\lambda_s)_{a_s} \cong_{\lie g}  V(\lambda_1) \otimes  \ldots \otimes V(\lambda_s).
$$
\begin{rem}
These fusion products play an important role in the construction of level one local Weyl modules for $\lie g \otimes \bc[t]$ (\cite{CL06,FoL07}).
\end{rem}

Since their introduction in \cite{FL99}, the following conjectures have remained open for the past 20 years:
\begin{conj}
Let $a_1, \ldots, a_s \in \bc$ pairwise distinct and $\lambda_1, \ldots, \lambda_s \in P^+$.
\begin{enumerate}
\item $V(\lambda_1)_{a_1} \ast \ldots \ast V(\lambda_s)_{a_s}$ is independent of the parameters  $a_1, \ldots, a_s$.
\item The fusion product, defined for any finite collection of cyclic modules, is associative, e.g. 
\[
\left(V(\lambda_1)_{a_1} \ast V(\lambda_2)_{a_2}
\right)_{b_1} \ast V(\lambda_3)_{a_3} \cong V(\lambda_1)_{a_1} \ast \left(V(\lambda_2)_{a_2} \ast V(\lambda_3)_{a_3}
\right)_{b_1} .
\]
\end{enumerate}
\end{conj} 
\bigskip

\subsection{Fusion product with two simple factors}

We consider here the case of the fusion product of two evaluation modules. Fix $\lambda_1, \lambda_2 \in P^+$, set $\lambda = \lambda_1 + \lambda_2$ and denote $I(\lambda_1, \lambda_2)$ the left ideal in $U(\lie g \otimes \bc[t])$ generated by
\[
\lie n^+ \otimes \bc[t],\quad
h_{\alpha} - \lambda(h_\alpha),\quad
\lie h \otimes t \bc[t]
\]
and
\[
f_{\alpha}^{\lambda(h_\alpha)+1},\quad
(f_\alpha \otimes t)^{\min\{ \lambda_1(h_\alpha),   
\,\lambda_2(h_\alpha)\} +1},\quad
\lie n^- \otimes t^2 \bc[t],
\]
where $\alpha$ ranges over all positive roots.
The following conjecture is due to E.~Feigin \cite{Fou15} and we will discuss its implications to Schur positivity and truncated local Weyl modules in Section~\ref{sec-conj}:
\begin{conj}\label{conj-fusion}
Let $a_1 \ne a_2 \in \bc$, then there is an isomorphism of graded $\lie g \otimes \bc[t]$-modules
\[
V(\lambda_1)_{a_1} \ast V(\lambda_2)_{a_2} \cong U(\lie g \otimes \bc[t])/I(\lambda_1,  \lambda_2).
\]
\end{conj}

\begin{rem} A few remarks need to be made.
\begin{enumerate}
    \item This conjecture has been proved for $\lie{sl}_2$ in \cite{FF02}, but we provide a different proof in the current paper.
    \item The conjecture has been proved for $\lie{sl}_n$ and $\lambda_1 \gg \lambda_2$, i.e. $\lambda_1 + \text{weights } (V(\lambda_2)) \subset P^+$, in \cite{FFL11}. The proof uses a new type of monomial bases for $V(\lambda)$ and initiated the framework on PBW degenerations (\cite{FFL11}).
    \item Various cases such as multiples of fundamental weights in the $\lie{sl}_n$-case are discussed in \cite{Fou15}.
\end{enumerate}
\end{rem}

We are aiming to prove this conjecture and make a first step towards a proof by reformulating the conjecture into the language of Gr\"obner bases. We will use the following proposition while omitting the obvious proofs.

\begin{prop}\label{prop-fus-para}
Let $a \in \bc$, then $x \otimes t \mapsto x \otimes (t-a)$ induces an automorphism $\phi_a$ of of $U(\lie g \otimes \bc[t])$. For $a_1 \neq a_2 \in \bc$, $\lambda_1, \lambda_2 \in P^+$, using the pullback gives
$$
\phi_{a_1}^* V(\lambda_1, \lambda_2, a_1, a_2) \cong V(\lambda_1, \lambda_2, 0, a_2-a_1).
$$
\end{prop}

This allows us to restrict the analysis of fusion products to parameters of the form $(0,a) \in \{0 \} \times \bc^*$. 
\begin{lem}\label{lem-gen-fus}
Let $\lambda_1, \lambda_2 \in P^+$, $a \in \bc^*$, then $V(\lambda_1, \lambda_2, 0, a)$ is the $U(\lie g \otimes \bc[t])$-module 
presented via the left ideal $I_a(\lambda_1, \lambda_2)$, which is  generated by
\[
\lie n^+ \otimes \bc[t],\quad
h - (\lambda_1 + \lambda_2)(h),\quad
h \otimes t - a \lambda_2(h),
\]
for all $h \in \lie h$, and
\[
x \otimes t^2 - ax \otimes t,\quad
f_{\alpha}^{(\lambda_1 + \lambda_2)(h_\alpha) +1 },\quad
(f_{\alpha} \otimes t)^{ \lambda_2(h_\alpha) +1}, \quad (f_{\alpha} \otimes (t-a))^{ \lambda_1(h_\alpha) +1},
\]
for all $x \in \lie g$ and positive roots $\alpha \in R^+$.
\end{lem}
\begin{proof}
First of all, we notice that $V(\lambda_1, \lambda_2, 0, a)$ is a quotient of $U(\lie g \otimes \bc[t])/ I_a(\lambda_1, \lambda_2)$. We have a closer look at the latter module. $U(\lie g \otimes \bc[t])/I_a(\lambda_1, \lambda_2)$ is in fact a cyclic highest weight module with one-dimensional highest weight space, due to the relations
\[
\lie n^+ \otimes \bc[t],\quad
h - (\lambda_1 + \lambda_2)(h),\quad
h \otimes t - a \lambda_2(h).
\]
This implies that $U(\lie g \otimes \bc[t])/ I_a(\lambda_1, \lambda_2)$ is a quotient of the local Weyl module $W_0(\lambda_1) \otimes W_a(\lambda_2)$ (see \cite{CP01} or Section~\ref{sec-weyl} for more details). 
So it is a module for $\lie g \otimes \bc[t]/t^N \oplus \lie g \otimes \bc[t]/(t-a)^N$ for some $N > 0$. Using 
\[x \otimes t^2 - ax \otimes t = 0 \text{ for all } x \in \lie g,
\]
we can set $N = 1$ and hence $U(\lie g \otimes \bc[t])/ I_a(\lambda_1, \lambda_2)$ is a tensor product of evaluation modules.

\end{proof}

This gives a one-parameter family of left ideals in $U(\lie g \otimes \bc[t])$ and hence we are able to apply methods from the theory of Gr\"obner bases  (for the setup, we refer to Section~\ref{sec-grobner}):
\begin{conj}\label{conj-flat}
$I_a(\lambda_1, \lambda_2)$ is a flat family of left  ideals (over $\bc[a]$) in $U(\lie g \otimes \bc[t])$ and there exists a monomial ordering on $U(\lie g \otimes \bc[t])$ such that the leading term ideals of $I_a(\lambda_1, \lambda_2)$ and $I(\lambda_1, \lambda_2)$ coincide.
\end{conj}

To prove the conjecture about the defining relations for the fusion products it is enough to prove the conjecture on the Gr\"obner basis:
\begin{thm}\label{thm-main2}
Conjecture~\ref{conj-flat} implies Conjecture~\ref{conj-fusion}.
\end{thm}
\begin{proof}
By definition of the ideal $I(\lambda_1, \lambda_2)$, there is a surjective map of $\lie g \otimes \bc[t]$-modules
$$
U(\lie g \otimes \bc[t])/I(\lambda_1, \lambda_2) \twoheadrightarrow V(\lambda_1)_{a_1} \ast V(\lambda_2)_{a_2},
$$
Conjecture~\ref{conj-flat} implies that 
\[
\dim U(\lie g \otimes \bc[t])/I(\lambda_1, \lambda_2) = \dim U(\lie g \otimes \bc[t])/I_a(\lambda_1, \lambda_2)
\]
but then with Proposition~\ref{prop-fus-para} and Lemma~\ref{lem-gen-fus}, we have
\[
 \dim U(\lie g \otimes \bc[t])/I_a(\lambda_1, \lambda_2) = \dim V(\lambda_1)_{a_1} \ast V(\lambda_2)_{a_2}.
 \]
Hence the surjective map is in fact an isomorphism.
\end{proof}

\section{Gr\"obner Theory in \texorpdfstring{$G$}{G}-algebras - Setup}\label{sec-grobner}
We recall in this section the Gr\"obner theory for $G$-algebras and how this applies to our setup. In fact, we first degenerate our $G$-algebra using a two-sided ideal and at the same time, using the same parameter, we degenerate a left ideal of the $G$-algebra.
\subsection{\texorpdfstring{$G$}{G}-algebras and Gr\"obner bases}
A total ordering $\leq$ on the monoid $(\bn_0^n,+,0)$ is called \emph{admissible} if $\alpha\leq\beta$ implies $\alpha+\gamma\leq\beta+\gamma$ for all $\alpha,\beta,\gamma\in\bn_0^n$.
Let $K$ be a field and $A$ a $K$-algebra generated by $x_1,\ldots,x_n$.
\begin{itemize}
\item The set of \emph{standard monomials} of $A$ is
\[
\Mon(A):=\{x^\alpha\mid\alpha\in\bn_0^n\}
:=\{x_1^{\alpha_1}x_2^{\alpha_2}\cdots x_n^{\alpha_n}\mid\alpha_i\in\bn_0\}.
\]
\item
Let $\leq$ be an admissible total ordering on $\bn_0^n$.
Any $f\in\text{$K$-span}(\Mon(A))\setminus\{0\}$ has a unique representation $f=\sum_{\alpha\in\bn_0^n}^{}c_\alpha x^\alpha$ with $c_\alpha\in K$, where $c_\alpha=0$ for almost all $\alpha$.
Now we define
\begin{itemize}
\item $\lexp(f):=\max\{\alpha\in\bn_0^n\mid c_\alpha\neq 0\}$, the \emph{leading exponent} of $f$ with respect to $\leq$,
\item $\lc(f):=c_{\lexp(f)}\in K\setminus\{0\}$, the \emph{leading coefficient} of $f$ with respect to $\leq$,
\item $\lm(f):=x^{\lexp(f)}\in\Mon(A)$, the \emph{leading monomial} of $f$ with respect to $\leq$.
\end{itemize}
\end{itemize}

For $n\in\bn$ and $1\leq i<j\leq n$ consider the constants $q_{ij} \in K\setminus\{0\}$ and
polynomials $d_{ij}\in K[x_1,\ldots,x_n]$.
Suppose that there exists an admissible ordering $\leq$ on $\bn_0^n$ such that for any $1\leq i<j\leq n$ either $d_{ij}=0$ or $\lexp(d_{ij})\leq\lexp(x_ix_j)$ holds.
The $K$-algebra
\[
A:=K\langle x_1,\ldots,x_n\mid \{ x_j x_i = q_{ij} x_i x_j + d_{ij}:1\leq i<j\leq n\}\rangle
\]
is called a \emph{$G$-algebra} if $\Mon(A)$ is a $K$-basis of $A$\footnote{which is equivalently formulated via algebraic relations between $q_{ij}$ and $d_{ij}$ \cite{LS03, LVdiss}, generalizing the Jacobi identities for commutators}. It is additionally \emph{of Lie type}, if all $q_{ij}=1$.

$G$-algebras (\cite{LS03, LVdiss}) are also known as algebras of solvable type and as PBW-algebras; they are left and right Noetherian domains. 
As an important example, universal enveloping algebras of finite-dimensional Lie algebras over arbitrary fields are $G$-algebras of Lie type.

Let $A$ be a $G$-algebra with the fixed monomial ordering $\prec$. 
For a left ideal $I \subset A$, a subset $G\subset I$ is a \emph{left Gr\"obner basis} of $I$, 
if 
for all $f \in I$ there exists $g\in G$ such that 
$\lexp(g) \leq_{cw} \lexp(f)$ (componentwise comparison). Over a $G$-algebra, one always finds a \emph{finite} left Gr\"obner basis. It is usually constructed by means of a \emph{generalized Buchberger's algorithm} \cite{LS03, LVdiss}, in which 
a set of \emph{critical pairs} is formed out of starting elements. Then, for each pair, the $s$-polynomial is formed, which is then \emph{reduced} to the remainder of the left division algorithm. If the remainder is nonzero, it is added to the starting set. The algorithm terminates when all critical pairs are reduced to zero.

\begin{rem}
\label{rem-leading-ideal}
The {\bf ideal of leading terms} 
over a $G$-algebra $A$ in, say, variables $x_1, \ldots, x_m$ over a field $K$ needs to be defined quite differently from the way it arises in the commutative case \cite{LVdiss}. Namely, let us collect leading exponents (with respect a fixed monomial ordering $\prec$) of polynomials from a left ideal $L \subset A$ into a set $Exp(L)$. It has a natural structure of a monoid ideal in $\mathbb{N}_0^m$, which is necessarily finitely generated by Dixon's lemma. Suppose the set of generators is $G(L)$. The ideal  
$Exp(L) = (G(L)) \subset \mathbb{N}_0^m$ is bijectively translated into the finitely generated monomial ideal 
$LT(L) := (x^{\alpha}: \alpha \in G(L))$ in the commutative ring $K[x_1, \ldots, x_m]$. We call
$LT(L) \subset K[x_1, \ldots, x_m]$
{\bf the ideal of leading terms of}
$L\subset A$.
Notably, the commutative ring $K[x_1, \ldots, x_m]$ need not be the associated graded algebra of $A$, though sometimes it is. A mini-example showing the necessity of this construction is as follows:
let $L = (x, x\partial +1 )$ in the first Weyl algebra $A=K\langle x, \partial \mid \partial x = x\partial +1\rangle$: the construction we've described gives $(x, x\partial) \in K[x,\partial]$ with $x$ as its' minimal basis, while the ideal of leading terms directly in $A$ is
$(x, x \partial) = (1)$ since 
$1 = \partial x - x \partial$ is in the ideal.
\end{rem}

Note that over $G$-algebras, finite right and two-sided Gr\"obner bases exist as well and are algorithmically computable.
All these and many other algorithms for the whole class of $G$-algebras and their factor-algebras are implemented in the freely available computer algebra system {\sc Singular} as a subsystem called {\sc Plural} \cite{LS03}.

For using Gr\"obner bases over $G$-algebras in theoretical arguments, the following \emph{Generalized Product Criterion} \cite{LS03, LVdiss} is very useful:
\begin{lem}
\label{gpc}%
Let $A$ be a $G$-algebra of Lie type
and $f, g \in A$. Suppose that $\lm(f)$ and $\lm(g)$ have no common factors, then the $s$-polynomial of $f$ and $g$ reduces to $[g,f]$ with respect to $\{f,g\}$.
\end{lem}

That is, in the situation described in 
\Cref{gpc}, 
we can compute Lie brackets instead of some $s$-polynomials. 
However, note that in general
we cannot compute a Gr\"obner basis
of a left ideal by merely computing Lie brackets.

\subsection{Gr\"obner degenerations for current algebras}
We want to apply the theory of $G$-algebras to current algebras (see \Cref{sec-setup}).

\begin{lem}\label{lem-grobner-bi}
For $m\in\mathbb{N}$ let $p = t^m - \sum_{i=0}^{m-1} p_i t^i \in \bc[t]$ be a polynomial. Then $U\left(\lie{g} \otimes \bc[t]/(p)\right)$ is a $G$-algebra.
\end{lem}

\begin{proof}
Let $\{x_1, \ldots, x_n\}$ be a basis of $\lie{g}$. Denote by $x_{ij}:= x_i \otimes t^j$ for $1 \leq i \leq n$ and  $0 \leq j <m$. In the free algebra $\bc\langle \{ x_{ij} \} \rangle$, we derive the following relations 
$1 \leq i,k \leq n$ and  $0 \leq j,\ell <m$:
\begin{itemize}
\item By definition of the current algebras, $[x_{ij},x_{k\ell}] = [x_i,x_k]\otimes t^{j+\ell}$. 
\item In particular, $[x_{ij}, x_{i\ell}]=0$.
\item Suppose that $[x_i,x_k] = \sum_r c^{ik}_r x_r \neq 0$, then
\[
[x_i,x_k]\otimes t^{j+\ell}
= \begin{cases}
\sum_{r=1}^n c^{ik}_r x_{r,j+\ell} & \text{if } j+\ell < m, \\
\sum_{r=1}^n 
\sum_{s=0}^{m-1} 
c^{ik}_r d_{j+\ell,s} x_{rs}
& \text{if } m \leq j+\ell < 2m-1, \\
\end{cases}
\]
where $\sum_{s=0}^{m-1}
d_{j+\ell,s} t^s = t^{j+\ell} 
\mod p$.
\end{itemize}
As we can see, the relation form a two-sided Gr\"obner basis of a $G$-algebra type, since any triple commutator of, say $x_{ij}, x_{k\ell}, x_{uw}$, reduces to the triple commutator of $x_i,x_k,x_u$, which is subject to the Jacobi identity:
\[
[[x_{ij}, x_{k\ell}], x_{uw}] = [[x_i,x_k]\otimes t^{j+\ell}, x_u \otimes t^{w}]= [[x_i,x_k],x_u]\otimes t^{j+\ell+w}.
\]
\end{proof}
As we can see, the proof carries verbatim to the case of an arbitrary field; also to the case, where we allow the coefficients $p_i$ of $p$ to be from $\bc[a_1,\ldots,a_q]$
instead of just $\bc$ for indeterminants $a_i$ (which commute with elements from $\lie{g}$).

\subsection{Two degenerations at once}
We outline here how we restrict from $\lie g \otimes \bc[t]$ to finite-dimensional Lie algebras and why it is enough to consider Gr\"obner basis for the ideals $I_a( \lambda_1, \lambda_2)$.
\bigskip

For $a \in \bc$, we consider the ideal $I_a \subset U(\lie g \otimes \bc[t])$ generated by $\lie g \otimes (t^2 - at)$. Then 
\[
U(\lie g \otimes \bc[t])/I_a \cong U(\lie g \otimes \bc[t]/ (\lie g \otimes (t^2 - at))) \cong U(\lie g \otimes \bc[t]/(t^2 -at)).
\]
Consider a monomial well-ordering $\succ$ on the countably generated algebra $U(\lie g \otimes \bc[t])$, compatible with  the natural degree on $\bc[t]$.
Then $\bc[t]/(t^2 - at)$ defines a flat family, since the ideal of leading terms with respect to $\succ$ 
is independent on $a$.
Hence the ideal of leading terms of $I_a$ is in fact generated by $\lie g \otimes t^2$, it is again independent on $a$, 
and thus $I_a$ defines a Gr\"obner degeneration. So $U(\lie g \otimes \bc[t])/I_a $ is a flat family of $\bc[a]$-modules. \bigskip

 Consider now a left ideal $J_a \subset U(\lie g \otimes \bc[t])/I_a$ which admits a left Gr\"obner basis (with respect to a monomial ordering, being a restriction of $\succ$ above), such that the left ideal of leading terms of $J_a$ is 
 independent on $a$. Combining this with the finite generation of 
 $(U(\lie g \otimes \bc[t])/I_a)$ as an algebra, we 
 invoke a classical Gr\"obner argument: the   Gel'fand-Kirillov dimension (or the Krull dimension in the  commutative case), the $\bc$-dimension 
 and the rank over $\bc[a]$ of $(U(\lie g \otimes \bc[t])/I_a)/J_a$
 are all the same for any value of $a$. In other words, we have just proved
\begin{thm}\label{thm-2at1}
 Viewed as $\bc[a]$-modules, $\left(U(\lie g \otimes \bc[t])/I_a\right)/J_a$ is a flat family, \\ i.e.~$U(\lie g \otimes \bc[t]/(t^2))/J_0$ is the special fiber of a Gr\"obner degeneration.
 \end{thm}


\section{Schur positivity and local Weyl modules}\label{sec-conj}
In this section, we would like to explain the implications of Conjecture~\ref{conj-fusion}. These implications might be well known for experts but the connection will not be present for the general audience, so we include this section for the sake of completeness. We start with briefly recalling Schur positivity and local Weyl modules for truncated current algebras.

\subsection{Conjecture on Schur positivity}\label{sec-schur}
Let $\lambda \in P^+$, then following \cite{CFS12} we set
\[
\mathcal{P}_{\lambda} := \{ (\lambda_1, \lambda_2) \in P^+ \times P^+ \mid \lambda_1 + \lambda_2 = \lambda \}
\]
and define a partial order on $\mathcal{P}_\lambda$
\[
(\lambda_1, \lambda_2) \succeq (\mu_1, \mu_2)
\,
:\Longleftrightarrow
\,
(\lambda_1 - \lambda_2)(h_\alpha) \leq (\mu_1 - \mu_2)(h_\alpha) \; \forall \, \alpha \in R^+.
\]
This relation is equivalent to 
\[
\min \{ \lambda_1(h_\alpha), \lambda_2(h_\alpha) \} \leq \min \{\mu_1(h_\alpha), \mu_2(h_\alpha)\} \; \forall\, \alpha \in R^+.
\]
The cover relations for $(\mathcal{P}_\lambda, \succeq)$ have been studied in \cite{CFS12}. 
It is important to notice 
that there exists, up to $S_2$-symmetry, a unique maximal element in the poset,  $(\lambda^{\max, 1}, \lambda^{\max,2})$. This is obtained by splitting $\lambda$ into two weights which differ by an alternating sum of fundamental weights only.

For the connection to Schur positivity, we briefly recall Schur polynomials and characters:\\
For $\lambda \in P^+$, we denote $s_\lambda = \operatorname{char } V(\lambda)$, which is equal to the Schur polynomial of the partition $\lambda$. Relevant for our purpose might be 
that for $(\lambda_1, \lambda_2) \in \mathcal{P}_{\lambda}$
$$\operatorname{char}_{\lie{sl}_n} V(\lambda_1)_{a_1} \ast V(\lambda_2)_{a_2} = s_{\lambda_1} s_{\lambda_2}.
$$
The Schur polynomial form a basis of the ring of symmetric polynomials in $n$ variables (modulo the relation $x_1 \ldots x_n -1$) and we call a symmetric polynomial \textit{Schur positive} if it is a non-negative linear combination of Schur polynomials. The interesting implication for representation theory is that any Schur positive symmetric polynomial is the character of a finite-dimensional $\lie{sl}_n$-module. \\

The following conjecture has been stated in this generality for the first time in \cite{DP07} and was later rediscovered in \cite{CFS12}. 
\begin{conj}\label{conj-schur}
Let $\lambda \in P^+$ and suppose $(\lambda_1, \lambda_2) \succeq (\mu_1, \mu_2)$ in $(\mathcal{P}_\lambda, \succeq)$, then $s_{\lambda_1} s_{\lambda_2} - s_{\mu_1} s_{\mu_2}$ is Schur positive.
\end{conj}

Some remarks on this conjecture and what is known so far
\begin{rem}
\begin{enumerate}
    \item The conjecture is a generalization (in some direction) of the row shuffle conjecture, e.g. for the unique maximal element $(\lambda^{max,1}, \lambda^{max,2})$ in $(\mathcal{P}_\lambda, \succeq)$ one has that
    $$
    s_{\lambda^{max,1}} \;  s_{\lambda^{max,2}} - s_{\mu_1} s_{\mu_2}
    $$
    is Schur positive. This has been conjectured by \cite{FFLP05} and proved by \cite{LPP07}.
    \item In \cite{DP07} it has been shown that the support (in terms of the basis of Schur polynomials) of $s_{\mu_1} s_{\mu_2}$ is contained in the support of $s_{\lambda_1} s_{\lambda_2}$ whenever $(\lambda_1, \lambda_2) \succeq (\mu_1, \mu_2)$ in $\mathcal{P}_{\lambda}$.
    \item The conjecture would imply that if $(\lambda_1, \lambda_2) \succeq (\mu_1, \mu_2)$ in $\mathcal{P}_\lambda$, then there is a surjective map of $\lie{sl}_n$-modules $V(\lambda_1) \otimes V(\lambda_2) \twoheadrightarrow V(\mu_1) \otimes V(\mu_2)$.
    \item It has been shown in \cite{CFS12} that  $$ \dim V(\lambda_1) \otimes V(\lambda_2) \geq \dim V(\mu_1) \otimes V(\mu_2),$$
    if $(\lambda_1, \lambda_2) \succeq (\mu_1, \mu_2)$ in $(\mathcal{P}_\lambda, \succeq)$.
    \item The conjecture has been proved for $\lie{sl}_3$ in \cite{CFS12} using combinatorics of crystal graphs.
\end{enumerate}
\end{rem}


\subsection{Conjecture on truncated local Weyl modules}\label{sec-weyl}
We briefly recall local Weyl modules for generalized current algebras and what is conjectured for truncated current algebras.
\bigskip

Let $A$ be quotient of the polynomial ring in finitely many variables by an homogeneous ideal and we denote $A_+$ the positively graded part of $A$. One can define then analogously to $A= \bc[t]$ the current algebra 
\[
\lie{sl}_n \otimes A.
\]

\begin{defn}
Let $\lambda \in P^+$, then the \textit{local Weyl module in the origin} $W_0(\lambda)$ is the $\lie{sl}_n \otimes A$-module generated by $w_\lambda$ with defining relations
\[
\lie{n}^+ \otimes A.w_\lambda = 0 ; \; \lie{h} \otimes A_+. w_\lambda = 0; \; h - \lambda(h).w_\lambda = 0; \; f_\alpha^{\lambda(h_\alpha) + 1}.w_\lambda = 0.
\]
for all $h \in \lie{h}$ and $\alpha \in R^+$.
\end{defn}
In fact these local Weyl modules can be defined way more general for any finitely generated, commutative, unital algebra and any maximal ideal in $\mathbb{A}_\lambda$ (see \cite{CFK10} for details), but we restrict ourselves to the unique graded local Weyl module. Several remarks should be made:
\begin{rem}
\begin{enumerate}
\item Local Weyl modules have been originally defined in \cite{CP01} for $A= \bc[t]$ and $A= \bc[t^{\pm}]$, then generalized in \cite{FL04} and \cite{CFK10}.
\item Let $A = \bc[t]$ or $\bc[t^{\pm}]$ and $\lambda = \sum m_i \omega_i$ then \[ W_0(\lambda) \cong_{\lie{sl}_n} V(\omega_1)^{\otimes m_1} \otimes \ldots \otimes V(\omega_{n-1})^{\otimes m_{n-1}},\] a result due to \cite{CL06} and later to \cite{FoL07}.
\item Beyond the two cases, nothing general is known about the $\lie{sl_n}$-structure of local Weyl modules. The formulas for the polynomial ring in three or more variables are already way more complicated even for $\lie{sl}_2$.
\end{enumerate}
\end{rem}

Since the generalization for more than one variable is out of reach at the moment, the natural question is to look for finite-dimensional quotients of $\bc[t]$. So for $k \geq 1$, we consider $A = \bc[t]/(t^k)$ and denote the local Weyl module $W_0(\lambda,k)$. 
\begin{rem}
Let $k = 1$, then obviously $W_0(\lambda,1) \cong V(\lambda)$.
\end{rem}
So let us consider the first non-trivial case and already here, the local Weyl modules are still not understood satisfactorily. The following conjecture is due to V.~Chari and appeared in \cite{Fou15}:

\begin{conj}\label{conj-weyl}
Let $\lambda \in P^+$
and let $(\lambda^{\max, 1}, \lambda^{\max,2})$ be the unique maximal element in $\mathcal{P}_{\lambda}$. Then for any $a_1 \neq a_2 \in \bc$:
\[
W_0 (\lambda, 2) \cong V(\lambda^{\max, 1})_{a_1} \ast V( \lambda^{\max,2})_{a_2}.
\]
\end{conj}

\begin{rem}
This conjecture is proved in several cases, for example for $\lie{sl}_2$ it follows from \cite{FF02}, various other cases have been proved in \cite{KL14}.
\end{rem}

\subsection{Flat degenerations and the two conjectures}

\begin{thm}\label{thm-main}
Conjecture~\ref{conj-flat} implies the conjecture on Schur positivity (Conjecture~\ref{conj-schur}) and the conjecture on local Weyl modules for truncated current algebras (Conjecture~\ref{conj-weyl}).
\end{thm}
\begin{proof}

We have seen already that the conjecture on the flat family implies the conjecture on the defining ideal for the fusion product. Suppose Conjecture~\ref{conj-flat} is true, then we have defining relations for the fusion product. First, we prove the implication of the conjecture on local Weyl modules: \\
We know by the universal property of the local Weyl module that there is a surjective map of $\lie{sl}_n \otimes \bc[t]/(t^2)$-modules
    \[
    W_0(\lambda,2) \twoheadrightarrow V(\lambda^{\max_1})_{a_1} \ast V(\lambda^{\max_2})_{a_2}.
    \]
    It remains to prove that the local Weyl module satisfies the defining relations of the fusion product. 
    This is true for $\lie{sl_2}$, as in this case the local Weyl module is explicitly computed by \cite{FF02} and its dimension coincides with the dimension of the right hand side.
    This implies that the positive root $\alpha$ of $\lie{sl}_2$, one has \[(f_\alpha \otimes t)^{\min\{ \lambda^{max_1}(h_\alpha), \, \lambda^{max_2}(h_\alpha)\}+1} = 0\] in the local Weyl module.
    Now, we consider $\lie{sl}_n$ and a positive root $\alpha$ of $\lie{sl}_n$. Let $W_0(\lambda,2)$ be the local Weyl module, then we consider the submodule within $W_0(\lambda,2)$, generated through the unique highest weight space by $\lie{sl}(\alpha) \otimes \bc[t]/(t^2)$, denote this submodule $V$. By construction, $V$ is a quotient of $W_0(\lambda(h_\alpha), 2)$, the local $\lie{sl}(\alpha) \otimes \bc[t]/(t^2)$-Weyl module and hence the relation  \[(f_\alpha \otimes t)^{\min\{ \lambda^{max_1}(h_\alpha), \, \lambda^{max_2}(h_\alpha)\}+1} = 0\] is true in the submodule and hence on the generator of $W_0(\lambda,2)$. The remaining relations of the fusion product are obviously satisfied. We conclude that $W_0(\lambda,2)$ is a quotient of  $V(\lambda^{\max_1})_{a_1} \ast V(\lambda^{\max_2})_{a_2}$ and hence they are isomorphic.\\
We turn to the implication of the conjecture on Schur positivity:\\
{\it Proof for Schur positivity}: $(\lambda_1, \lambda_2) \preceq (\mu_1, \mu_2)$, then $\lambda_1 + \lambda_2 = \mu_1 + \mu_2$ and
\[
\min \{ \lambda_1(h_\alpha), \lambda_2(h_\alpha) \} \leq \min \{\mu_1(h_\alpha), \mu_2(h_\alpha)\} \; \forall\, \alpha \in R^+.
\]
Suppose now Conjecture~\ref{conj-fusion} is true, then $I(\lambda_1, \lambda_2) \subset I(\mu_1, \mu_2)$ and there is a surjective map of $\lie{sl}_n \otimes \bc[t]$-modules
\[
V(\lambda_1)_{a_1} \ast V(\lambda_2)_{a_2} \twoheadrightarrow V(\mu_1)_{a_1} \ast V(\mu_2)_{a_2}.
\]

\end{proof}

\section{The \texorpdfstring{$\mathfrak{sl}_{2}(\bc)$}{sl2}-case}\label{sec-sl2}
In this section, we prove Conjecture~\ref{conj-flat} for $\lie g = \lie{sl}_2$. This implies of course Conjecture~\ref{conj-fusion} for $\lie{sl}_2$, which has been known before due to \cite{FF02}. We provide a new approach which might generalize to $\lie{sl}_n$. The idea is outlined in Theorem~\ref{thm-2at1}. When considering $\lie{sl}_2 \otimes \bc[t]/(t^2 - at)$, one has by Lemma~\ref{lem-grobner-bi} a Gr\"obner basis for the ideal generated by $\lie{sl}_2 \otimes (t^2 - at)$, since $\bc[t]$ is a principal ideal domain and hence $\{t^2 - at\}$ is a Gr\"obner basis for $(t^2  -at)$. By Theorem~\ref{thm-2at1} we are left to show that there is a Gr\"obner basis for the $U(\lie{sl}_2 \otimes \bc[t]/(t^2 - at))$-ideal $I_a(\lambda, \mu)$, cf. \Cref{lem-gen-fus}. The goal of this section is therefore the proof of Theorem~\ref{thm-grobner-left}, where such a Gr\"obner basis will be established explicitly.

\subsection{Commutator relations in tensor products}
We fix non-negative integers $\l\geq \m$ and we set for $x \in\mathfrak{sl}_2$:
\[
x_0 := x \otimes 1
\ ,
\quad 
x_1 := x \otimes t
\quad\text{in }\mathfrak{sl}_2\otimes\mathbb{C}[t]
\ ,
\]
so in particular, we have the elements $h_0,h_1,f_0,f_1,e_0,e_1\in\mathfrak{sl}_2\otimes\mathbb{C}[t]$.

In the following, we will prove commutation relations involving such elements and additional elements which will be denote by $F_i$, which will play a crucial role in the Gr\"obner basis we will describe.

\begin{defn} For all $0\leq j\leq i\leq \m+1$, we define
$$
m_i := \l+\m+1-i
\ ,\quad
c_{ji} :=\binom{m_i}{j} \binom{\m - j}{\m - i}
= \binom{m_i}{j} \binom{\m - j}{i - j}
  \ ,\quad
F_i
:= \sum_{k=0}^{i} c_{ki} (-a)^{i-k} f_1^{k} f_0^{m_i - k }
$$
(note that $c_{ki}=0$ as soon as $m_i<k$, so no negative powers of $f_0$ are occurring) and for all $0\leq i\leq \m$,
$$
p_i := -2 m_i (i+1) / (m_i-i-1) 
,\quad
q_i := -2 m_i (\m-i) / (m_i-i-1)  
\ .
$$
unless $\l=\m=i$, in which case we define $p_\m:=0$ and $q_\m:=0$.
\end{defn}

We record that $F_{\m+1} = \binom{\l}{\m+1} f_1^{\m+1} f_0^{\l-\m-1}$ is a monomial multiple of $f_1^{\m+1}$, which is equal to zero if $\l=\m$. On the other end of the spectrum we have another monomial $F_0 = f_0^{\l+\m+1}$.

\begin{lem} For all $0\leq i\leq\m$,
\begin{equation}\label{eq-hFi}
[h_0, F_i] = -2 m_i F_i
\ ,\quad
[h_1, F_i] = p_i f_0 F_{i+1} + q_i a F_i
\ ,
\end{equation}
\begin{equation}\label{eq-hfmu}
[h_0, f_1^{\m+1}] = -2(\m+1) f_1^{\m+1}
\ ,\quad\text{and}\quad
[h_1, f_1^{\m+1}] = -2a(\m+1) f_1^{\m+1}
\ .
\end{equation}
\end{lem}

\begin{proof} For all $s,t\geq 0$, we have 
\begin{equation} \label{eq-comm-h-f1sf0t}
 [h_0, f_1^s f_0^t] 
= -2(s+t) f_1^s f_0^t
\quad\text{and}\quad
[h_1, f_1^s f_0^t] 
= -2as f_1^s f_0^t - 2t f_1^{s+1} f_0^{t-1} 
\ ,
\end{equation} 
which imply directly \Cref{eq-hfmu} and the formula for $[h_0,F_i]$ in \Cref{eq-hFi}. For the remaining formula for $[h_1,F_i]$, we compute
\begin{align*}
[h_1, F_i]
&= \sum_{j=0}^i
-2 a j c_{ji} (-a)^{i-j} 
f_1^j f_0^{m_i-j}
-2 (m_i-j) c_{ji} (-a)^{i-j} 
f_1^{j+1} f_0^{m_i-(j+1)}
\ ,
\end{align*}
so the non-zero coefficients of monomials of the form $f_1^j f_0^{m_i-j}$ are
$$
 C_j := \begin{cases} 
  -2(m_i-i) c_{ii} & j=i+1 \\
 -2 a j c_{ji} (-a)^{i-j} - 2(m_i-j+1) c_{j-1,i} (-a)^{i-j+1} & 1\leq j \leq  i \end{cases}
 \ .
$$
On the other hand, the corresponding coefficients in $q_i a F_i + p_i f_0 F_{i+1}$ are
$$
 C'_j := \begin{cases} 
  p_i c_{i+1,i+1} & j=i+1 \\
  q_i a c_{ji} (-a)^{i-j} + p_i c_{j,i+1} (-a)^{i+1-j}  & 0\leq j \leq  i
 \end{cases}
 \ .
$$
Now if $\l=\m=i$, then $C_{i+1}=C'_{i+1}=C_0=C'_0=0$. Otherwise, the terms $c_{i+1,i+1}$, $c_{0,i}$, and $(m_i-i-1)$ are non-zero, and $C_{i+1}=C'_{i+1}$ and $C_0=C'_0$ if and only if
\begin{align*}
p_i = -2(m_i-i) c_{ii} / c_{i+1,i+1} = -2 m_i (i+1) / (m_i-i-1) 
\\
\text{and}\quad
q_i = p_i c_{0,i+1} / c_{0,i} = -2 m_i (\m-i) / (m_i-i-1)  
\ ,
\end{align*}
which is indeed consistent with the way we defined $p_i$ and $q_i$. For all $1\leq j\leq i$, $C_j=C'_j$ if and only if
$$
2 j c_{ji} - 2(m_i-j+1) c_{j-1,i} = - q_i c_{ji} + p_i c_{j,i+1} 
\quad\ldots
$$
If $\l=\m=i$, both sides are zero, so the identity holds trivially, otherwise the term $(m_i-i-1)$ is non-zero and we can simplify the identity:
\begin{align*}
\ldots\Leftrightarrow
0 &= \tfrac{m_i(i+1)}{m_i-i-1} c_{j,i+1}  
  - (\tfrac{m_i(\m-i)}{m_i-i-1}-j) c_{ji} 
  - (m_i-j+1) c_{j-1,i} 
\\
 &= c_{ji} \left( \tfrac{m_i(i+1)}{m_i-i-1} \tfrac{(m_i-j)(\m-i)}{m_i(i-j+1)}
 - (\tfrac{m_i(\m-i)}{m_i-i-1}-j) 
 - (m_i-j+1) \tfrac{j(\m-j+1)}{(m_i-j+1)(i-j+1)}
 \right)
 \\
 &= c_{ji} \left( \tfrac{(i+1)(m_i-j)(\m-i) }{ (m_i-i-1)(i-j+1)}
 - (\tfrac{m_i(\m-i)}{m_i-i-1}-j) 
 - \tfrac{j(\m-j+1)}{i-j+1}
 \right)
 \ .
\end{align*}
Indeed, after multiplying with the non-zero term $(m_i-i-1)(i-j+1)/c_{ji}$, the right-hand side simplifies to
\begin{align*}
 &(i+1)(m_i-j)(\m-i)
 - m_i(\m-i)(i-j+1)
 + j(i-j+1)(m_i-i-1)
 - j(\m-j+1) (m_i-i-1)
 \\
 &= 
 (i+1)(m_i-j)(\m-i)
 - m_i(\m-i)(i-j+1)
 - j(\m-i)(m_i-i-1)
 \\
 &= (\m-i) ((i+1)(m_i-j)
 - m_i(i-j+1)
 - j(m_i-i-1)
 )
 \\
 &= (\m-i) (-(i+1)j
 + m_i j
 - j(m_i-i-1)
 )
 = 0
 \ ,
\end{align*}
as desired.
\end{proof}

\newcommand\pfi{\partial_{f_1}}
\newcommand\pfo{\partial_{f_0}}

\begin{lem} For all $s,t\geq0$,
\begin{align}\label{eq-comm-e0-f1-f0}
    [e_0, f_1^s f_0^t] 
    &= s f_1^{s-1} f_0^t (h_1 - a(s-1)) + t f_1^s f_0^{t-1} (h_0 - 2s - t + 1)
\\
\label{eq-comm-e1-f1-f0}
    [e_1, f_1^s f_0^t] 
    &= a s f_1^{s-1} f_0^t (h_1 - a(s-1)) + t f_1^s f_0^{t-1} (h_1-2as) 
 - t (t-1) f_1^{s+1} f_0^{t-2}
\end{align}
\end{lem}

\begin{proof}
Using basic identities in the current algebra $\mathfrak{sl}_2\otimes\mathbb{C}[t]$ and \Cref{eq-comm-h-f1sf0t}, we compute for all $s,t\geq 1$
\begin{align*}
    [e_0, f_0^t] 
    &= h_0 f_0^{t-1} + f_0 [e_0,f_0^{t-1}]
     = -2 (t-1) f_0^{t-1} + f_0^{t-1} h_0 + f_0 [e_0,f_0^{t-1}]
     \\
     &= \sum_{k=0}^{t-1} f_0^k (-2 (t-k-1) f_0^{t-k-1} + f_0^{t-k-1} h_0)
     = t f_0^{t-1} (h_0 - (t-1))
     \ ,\quad\text{so}
\end{align*}\begin{align*}
[e_0, f_1^s f_0^t]
 &= h_1 f_1^{s-1} f_0^t + f_1 [e_0, f_1^{s-1} f_0^t]
 \\
 &= -2a(s-1) f_1^{s-1} f_0^t -2 t f_1^s f_0^{t-1}
  + f_1^{s-1} f_0^t h_1 + f_1 [e_0, f_1^{s-1} f_0^t]
  \\
 &= f_1^{s-1} f_0^t (h_1 - 2a(s-1)) + f_1^s f_0^{t-1} (-2t) + f_1 [e_0, f_1^{s-1} f_0^t]
 \\
 &= \sum_{k=0}^{s-1} f_1^k (f_1^{s-k-1} f_0^t (h_1 - 2a(s-k-1)) + f_1^{s-k} f_0^{t-1} (-2t)) + f_1^s [e_0, f_0^t]
 \\
 &= s f_1^{s-1} f_0^t (h_1 - a(s-1)) + t f_1^s f_0^{t-1} (-2(s-1)) + t f_1^s f_0^{t-1} (h_0 - (t-1))
 \\
 &= s f_1^{s-1} f_0^t (h_1 - a(s-1)) + t f_1^s f_0^{t-1} (h_0 - 2(s-1) - (t-1))
 \ .
 \end{align*}
and it is easy to see that the equation even holds for $s,t\geq0$, where the negative powers of $f_1$ or $f_0$ do not pose a problem, since they occur with coefficient $0$.

Similarly, we obtain
\begin{align*}
    [e_1, f_0^t] 
    &= h_1 f_0^{t-1} + f_0 [e_1,f_0^{t-1}]
     = -2 (t-1) f_1 f_0^{t-2} + f_0^{t-1} h_1 + f_0 [e_1,f_0^{t-1}]
     \\
     &= \sum_{k=0}^{t-2} f_0^k (-2 (t-k-1) f_1 f_0^{t-k-2} + f_0^{t-k-1} h_1)
      + f_0^{t-1} [e_1, f_0]
      \\
     &= - t (t-1) f_1 f_0^{t-2} + t f_0^{t-1} h_1 
     \ ,\quad\text{so}
\end{align*}\begin{align*}
    [e_1, f_1^s f_0^t]
 &= a h_1 f_1^{s-1} f_0^t + f_1 [e_1, f_1^{s-1} f_0^t]
 \\
 &= -2 a^2 (s-1) f_1^{s-1} f_0^t -2 a t f_1^s f_0^{t-1}
  + a f_1^{s-1} f_0^t h_1 + f_1 [e_1, f_1^{s-1} f_0^t]
  \\
 &= a (f_1^{s-1} f_0^t (h_1 - 2a(s-1)) + f_1^s f_0^{t-1} (-2t)) + f_1 [e_1, f_1^{s-1} f_0^t]
 \\
 &= \sum_{k=0}^{s-1} a f_1^k (f_1^{s-k-1} f_0^t (h_1 - 2a(s-k-1)) 
  + f_1^{s-k} f_0^{t-1} (-2t))
 + f_1^s [e_1, f_0^t]
 \\
 &= a s f_1^{s-1} f_0^t (h_1 - a(s-1)) + t f_1^s f_0^{t-1} (h_1-2as) 
 - t (t-1) f_1^{s+1} f_0^{t-2}
 \ .
 \end{align*}
 \end{proof}

\begin{lem} For all $0\leq i\leq\m$,
\begin{align}\label{eq-e0Fi}
[e_0, F_i] 
&=  (\d_{f_0} F_i) (h_0-(\l+\m)) + (\d_{f_1} F_i) (h_1-\m a)
\ ,
\\
\label{eq-e1Fi}
[e_1, F_i] 
&=  (\d_{f_0} F_i + a \d_{f_1} F_i) (h_1-\m a) + (i+1) (m_i-1)F_{i+1}
\ ,
\end{align}
\begin{equation}
[e_0, f_1^{\m+1}] 
= (\m+1) f_1^\m (h_1-\m a)
\ ,\quad\text{and}\quad
[e_1, f_1^{\m+1}] 
= a (\m+1) f_1^\m (h_1-\m a)
\ .
\end{equation}
\end{lem}

\begin{proof} \Cref{eq-comm-e0-f1-f0} and \Cref{eq-comm-e1-f1-f0} directly imply the last two identities. \Cref{eq-comm-e0-f1-f0} also allows us to take the coefficients of $f_1^j f_0^{m_i-j-1}$ on the two sides of the  asserted commutator formula involving $e_0$ to see that it is equivalent to the identity
\begin{align*}
0 
&= c_{j+1,i} (-a)^{i-j-1} (j+1) (h_1 - aj - h_1 + a\m) 
\\ &\qquad\qquad+ c_{ji} (-a)^{i-j} (m_i-j) (h_0 - 2j - (m_i-j) + 1 - h_0 + (\l+\m) ) 
\\
 &= (-a)^{i-j} (-c_{j+1,i} (j+1) (\m-j) + c_{ji} (m_i-j) (i-j) )
 \\
 &= \begin{cases} 0 & j=\m \\ c_{ji} (-a)^{i-j} (-\tfrac{ (i-j)(m_i-j) }{ (j+1)(\m-j) } (j+1) (\m-j) + (m_i-j) (i-j) ) & \text{else} \end{cases}
\end{align*}
for all $0\leq j\leq i$, which is true in any case.

Similarly, again considering the coefficients of $f_1^j f_0^{m_i-j-1}$ and now using \Cref{eq-comm-e1-f1-f0}, the asserted commutator formula involving $e_1$ is equivalent to the identity
\begin{align}
0 &= c_{j+1,i} (-a)^{i-j-1} a (j+1) (h_1-aj) + c_{ji} (-a)^{i-j} (m_i-j) (h_1-2aj)
\notag\\
 &\qquad\qquad- c_{j-1,i} (-a)^{i-j+1} (m_i-j+1)(m_i-j)
  - c_{ji} (-a)^{i-j} (m_i-j) (h_1-a\m) 
  \notag\\
  &\qquad\qquad - c_{j+1,i} (-a)^{i-j-1} a (j+1) (h_1 - a\m) + (i+1) (m_i-1) c_{j,i+1} (-a)^{i-j+1}
  \quad\ldots
  \label{eq-pf-lemma-commutator-e}
\end{align}
which is true if $\m=j$, since then $\m=i$ and
\begin{align*}
\ldots 
&= -c_{\m+1,\m} (\m+1) (h_1-a\m) + c_{\m,\m} (m_\m-\m) (h_1-2a\m)
  - c_{\m-1,\m} (-a) (m_\m-\m+1)(m_\m-\m)
  \\
 &\qquad\qquad
  - c_{\m,\m} (m_\m-\m) (h_1-a\m) 
  + c_{\m+1,\m} (\m+1) (h_1 - a\m)
\\
&=  -a c_{\m,\m} (m_\m-\m) (\m 
- \tfrac{\m}{(m_\m-\m+1)} (m_\m-\m+1)
)
\ ,
\end{align*}
which is true. Otherwise, the terms $(\m-j)$ and $(m_i-j+1)$ are non-zero and we can simplify \Cref{eq-pf-lemma-commutator-e}
\begin{align*}
 \ldots &= c_{ji} (-a)^{i-j} (
  -\tfrac{(i-j)(m_i-j)}{(j+1)(\m-j)} a(j+1)(\m-j)
  + (m_i-j)a (\m-2j)
  \\
  &\qquad\qquad + a \tfrac{j(\m-j+1)}{(i-j+1)(m_i-j+1)} (m_i-j+1)(m_i-j)
  + (i+1) (m_i-1) (-a) \tfrac{ (m_i-j)(\m-i)  }{ (m_i-1)(i-j+1) }
  ) 
  \\
  &= c_{ji} (-a)^{i-j+1} (m_i-j) (
  (i-j)
   - (\m-2j)
  -  \tfrac{j(\m-j+1)}{i-j+1} 
  + (i+1) \tfrac{ \m-i  }{ i-j+1 }
  )
  \\
  &= c_{ji} (-a)^{i-j+1} (m_i-j) (\m-i) (
  -1
  -  \tfrac{j}{i-j+1} 
  + (i+1) \tfrac{ 1  }{ i-j+1 }
  )\ ,
\end{align*}
which is true, too.
\end{proof}



\subsection{Gr\"obner basis for \texorpdfstring{$I_a(\lambda, \mu)$}{the defining ideal}}
In order to to prove the Gr\"obner basis property
of the ideal $I_a(\lambda, \mu)$ 
defined in \Cref{lem-gen-fus},
we are left with understanding the interplay of the $F_i$ in the commutative ring $\bc(a)[f_0, f_1]$. From now on let $\succ$ 
be a monomial well-ordering 
on both $\{f_0, f_1\}$ and 
$\{e_0, e_1, h_0, h_1, f_0, f_1\}$
satisfying $f_1 \succ f_0$. 

\begin{prop}
\label{FdonotFactor}
Suppose that for $1 \leq i \leq \mu$ there is a factorization $F_i = G_1 G_2$ in $\bc(a)[f_0,f_1]$ with $G_1, G_2$ both not being constant, then $G_1,G_2 \notin I_a(\lambda, \mu)$.
\end{prop}
\begin{proof}
Suppose one has a non-trivial factorization, say $F_i  = G_1 G_2$. Since $F_i$ is homogeneous of degree $\l+\m+1-i$, $G_1, G_2$ are both homogeneous and of the form
\[
\sum_{k =0 }^p c_k f_1^k f_0^{\ell - k},
\]
for some $0 \leq p \leq i$, $p \leq \ell < \lambda + \mu +1 - i$ and some coefficients $c_k$. 
It is enough to show that the set 
\[
\{ f_1^k f_0^{\ell - k}. (v_\l \otimes v_\m) \in V(\l)_{0} \otimes V(\m)_{a} \mid 0 \leq k \leq p\}
\]
is linearly independent. One has 
\[
f_1^k f_0^{\ell - k}. v_\l \otimes v_\m 
= a^k f_0^{\ell- k}(v_\l \otimes f_0^k v_\m).
\]
Hence it remains to show that the set 
\[
\{ f_0^{\ell - k}(v_\l \otimes f_0^k v_\m) \mid 0 \leq k \leq p\}
\]
with $p \leq i \leq \mu$ and $p \leq \ell < \lambda + \mu +1 - i$ is linearly independent. The weight space $(V(\l) \otimes V(\m))_{\l+\m - 2\ell}$ has in fact dimension $\lambda + \mu + 1 - \ell > i \geq p$. Linear independence follows from that the fact the columns of the $ (\lambda + \mu + 1  - \ell) \times (p+1)$-matrix
\[
A = \left( \frac{(\ell - (r-1))!}{(\ell - (r-1) - (s-1))!}\right)_{r,s}
\]
are linearly independent. In fact, for $ p = \lambda + \mu - \ell$, the determinant of the square matrix equals $\prod_{q=0}^p q!$, up to a sign. This implies that the set is linearly independent and hence neither $G_1$ nor $G_2$ are in $I_a(\lambda,\mu)$.
\end{proof}

We remark, that by the similar reasoning the monomial element $F_{\mu+1}$ factorizes, and therefore will be replaced by $f_1^{\mu+1}$ in the ideal $I_a(\lambda,\mu)$.

\newcommand{\spoly}{\operatorname{spoly}}

Let us denote the s-polynomial of two polynomials  $f,g$ by $\spoly(f,g)$. 


\begin{prop}
\label{theFs}
Over a polynomial ring $R:=\bc(a)[f_0,f_1]$, the set
$\mathcal{F}:=\left\{ F_i(\lambda, \mu): 0 \leq i \leq \mu+1 \right\}$ is a Gr\"obner basis 
with respect to $\succ$ of the ideal it generates. 
\end{prop}

\begin{proof}
We introduce a convenient notation first: 
\[
F_k = \sum_{j=0}^{k} c_{jk} (-a)^{k-j} f_1^{j} f_0^{m_k - j } = 
(-a)^{k} f_0^{m_0 - k} 
\sum_{j=0}^{k} c_{jk} g^j,
\]
where $g:=\tfrac{f_1}{-a f_0}$ is of degree $0$. Note that after multiplication of the right hand side we do not have fractions in $f_0$ since we assume $\lambda \geq \mu$. 

The monomial ordering, preferring $f_1$ over $f_0$,
carries over to the new notation
by ordering polynomials in $g$ (which are of total degree 0) by their exponents.
$F_k$ rewritten in $f_0, g$ is a graded polynomial of degree $m_k=m_0-k$. 
For a natural $\ell$ in the admissible range we see that 
$\spoly(F_k, F_{k+\ell})$ is graded of degree $m_0-k$: 
\[
S:=f_0^{\ell} F_{k+\ell} - \kappa g^{\ell} F_k  
= 
(-a)^{k+\ell} f_0^{m_0 - k} 
\sum_{j=0}^{k+\ell} c_{j,k+\ell} g^{j}
- \kappa 
(-a)^{k} f_0^{m_0 - k} 
\sum_{j=0}^{k} c_{jk} g^{j+\ell}
\]
where $\kappa := (-a)^{\ell} \tfrac{c_{k+\ell, k+\ell}}{c_{k,k}}$ guarantees that the leading term cancels out. If $S=0$, we are done. Otherwise there exists $0 \leq q < k+\ell$ such that
$S= f_0^{m_0 - k} p_1(g)$
for a polynomial $p_1 \in \bc[g]$.
Writing similarly $F_k$ as $f_0^{m_0 - k} p_2(g)$, let $G:=\gcd(p_1,p_2)\in\bc[g]$. Since there exist $a,b\in\bc[g]$ with $G=ap_1 + bp_2$, also $f_0^{m_0 - k}G = a S + b F_k$ belongs to the graded subspace of degree $m_0 - k$. Now $G \mid p_2$ implies $f_0^{m_0 - k} G \mid F_k$. Since by \Cref{FdonotFactor}, $F_k$ cannot have proper factors, which belong to $I_a(\l,\m)$, 
it follows that $S$ and $F_k$ differ by a constant, say $\kappa'\in\bc\setminus\{0\}$. Thus
  \[
   f_0^{\ell} F_{k+\ell} - \kappa g^{\ell} F_k = \spoly(F_k, F_{k+\ell}) =\kappa' F_k,
 \]
 and rewriting back in $f_0, f_1$-notation
  \[
  (-a)^{\ell}f_0^{2\ell} F_{k+\ell} - \kappa f_1^{\ell} F_k =
 \kappa' (-a)^{\ell}f_0^{\ell} F_k
 \]
  shows that the original $\spoly(F_k, F_{k+\ell})$  reduces to zero with $F_k$ for any $\ell\geq 1$ in the admissible range, therefore  $\mathcal{F}$ is a Gr\"obner basis.
\end{proof}

\begin{thm}\label{thm-grobner-left}
Let $\lambda\geq  \mu \in \mathbb{N}_{\geq 0}$ be dominant weights for $\mathfrak{sl}_2$ and the monomial ordering $\succ$ be chosen as in the current section. Then the following is a Gr\"obner basis with respect to $\succ$ of the left ideal $I_a( \lambda, \mu)$ defining $V_0(\lambda) \otimes V_a(\mu)$ as a $\lie{sl}_2 \otimes \bc[t]/(t^2 - at)$-module:
\[
\left\{ e_0, e_1, h_0 - (\lambda + \mu), h_1 - \mu a \right\} \cup \left\{ F_i(\lambda, \mu) \mid i = 0, \ldots, \mu \right\} \cup \{ f_1^{\mu+1} \}
.
\]
\end{thm}
\begin{proof}
For short, we use again $F_i := F_i(\l,\m)$ and $I_a:=I_a(\l,\m)$. We prove  that 
\[
\left\{ e_0, e_1, h_0 - (\lambda + \mu), h_1 - \mu a \right\} \cup \left\{ F_i(\lambda, \mu) \mid i = 0, \ldots, \mu \right\}
\]
is a Gr\"obner basis, then clearly adding $f_1^{\mu+1}$ is still a Gr\"obner basis since $f_1^{\mu+1}  \in I_a$.
By definition of the ideal $I_a$, the elements 
\[
\{e_0, e_1 , h_0 - (\lambda + \mu), h_1 - \mu a, F_0, f_1^{\m+1} \}
\]
generate the ideal. Using \eqref{eq-e1Fi}, we see that $F_i\in I_a$ for all $1\leq i\leq\m$. So it is left to show that this set is in fact a Gr\"obner basis for $I_a$. We have seen already that the commutator of pairs of the elements, obeying the condition from the  Lemma \ref{gpc}, have left Gr\"obner presentations: for all $0\leq i\leq\m$,
\begin{enumerate}
\item $[e_0,e_1] = 0$.
\item $[e_0, h_0 - (\lambda + \mu)] = -2e_0$.
\item $[e_0, h_1 - \mu a] = -2e_1$.
\item $[e_0, F_i] =  (\partial_{f_0} F_i) (h_0-(\lambda+\mu)) + (\partial_{f_1} F_i) (h_1-\mu a)$. 
\item $[e_0, f_1^{\m+1}] = (\m+1) f_1^\m (h_1-\m a)$
\medskip
\item $[e_1, h_0 - (\lambda+\mu)] = -2e_1$.
\item $[e_1, h_1 - \mu a] = 0$. 
\item $[e_1, F_i] = (\partial_{f_0} F_i + a \partial_{f_1} F_i) (h_1-\mu a) + (i+1) (m_i-1) F_{i+1}$ 
\item $[e_1, f_1^{\m+1}] = a (\m+1) f_1^\m (h_1-\m a)$
\medskip
\item $[h_0 - (\lambda + \mu), h_1 - \mu a] = 0$.
\item $[h_0 - (\lambda + \mu), F_i] = -2 m_i F_i.$
\item $[h_0 - (\lambda + \mu), f_1^{\m+1}] = -2 (\m+1) f_1^{\m+1}.$ 
\medskip
\item $[h_1 - \mu a, F_i] = p_i f_0 F_{i+1} + q_i a F_i$.
\item $[h_1 - \mu a, f_1^{\m+1}] = -2a(\m+1) f_1^{\m+1}$. 
\end{enumerate}
(Recall that $F_{\m+1}$ is a monomial and a multiple of $f_1^{\m+1}$.)
\medskip

In fact, $(1), (2), (3), (5), (6), (9)$ are obvious, $(4), (7)$ are from \eqref{eq-e0Fi} and \eqref{eq-e1Fi}, $(10), (11)$ are from \eqref{eq-hFi}.

Indeed, replacing $F_{\m+1} = f_0^{\l-\m-1} f_1^{\m+1}$ by $f_1^{\m+1}$ does not violate the Gr\"obner property of the set $\mathcal{F}$:
as in the proof of \Cref{theFs}, let $0\leq k,\ell \leq\mu$ be such that $k+\ell = \mu+1$. Then there exist appropriate constants such that
 \[
 \left(\kappa f_1^{\ell} +  \kappa' (-a)^{\ell}f_0^{\ell}\right) \cdot F_k = 
  (-a)^{\ell}f_0^{2\ell} F_{\mu+1} =
  (-a)^{\ell}f_0^{3\ell-\m-1} \cdot f_1^{\m+1}.
 \]
Since the leftmost factor is not divisible by $f_0$, it follows that all the s-polynomials involving $F_{\mu+1}$ can be expressed as such involving $f_1^{\m+1}$ and thus reduced to zero.
Now, from \Cref{theFs} it follows 
 that the set $\mathcal{F}' = 
\mathcal{F}\cup\{f_1^{\m+1}\}$
is a left Gr\"obner basis of the left ideal, generated by $\mathcal{F}$ over the ring $\bc(a)[f_0,f_1]$, hence the same holds over the algebra $A$ since $\mathcal{F}$ does not involve variables, other than $f_0$ and $f_1$, and monomial orderings are compatible by the setup, hence we are done. 
\end{proof}


We conclude the paper with 
\begin{prop}
Conjecture~\ref{conj-flat} for $\lie{sl}_2$.
\end{prop}
\begin{proof} Let $\lambda \geq \mu$. 
Using the results Theorem~\ref{thm-grobner-left} together with Lemma~\ref{lem-grobner-bi} and Theorem~\ref{thm-2at1}, we see that $I_a(\lambda, \mu)$ defines a flat family of ideals in $U(\lie g \otimes \bc[t])$. It remains to show that $I_0(\lambda, \mu) = I(\lambda, \mu)$. 
From the Gr\"obner basis in  \Cref{thm-grobner-left}
we read off the generators of $I_0(\lambda, \mu)$:
\[
\{ e_0, e_1, h - (\lambda + \mu), h_1, f_1^{\mu+1} \} \cup \{  f_1^kf_0^{\lambda + \mu +1 - 2k} \mid k = 0, \ldots, \mu \},
\]
additionally one has $\lie{sl}_2 \otimes t^2 \bc[t]$ as generators. Comparing with the generators of $I(\lambda, \mu)$ we have obviously
\[
I(\lambda, \mu) \subseteq I_0(\lambda, \mu)
\]
and it remains to show that $f_1^kf_0^{\lambda + \mu +1 - 2k}  \in I(\lambda, \mu)$ but this follows from
\[
(\operatorname{ad}_{e_1})^{k} f_0^{\lambda + \mu+1} = \left(2^k \prod_{\ell = 0}^{k-1} \binom{ \lambda + \mu + 1 - 2\ell}{2}\right) f_1^{k} f_0^{\lambda + \mu +1 - 2k} + J 
\]
where $J$ is the left ideal generated by $\{ e_1, h_1, \lie{sl}_2 \otimes t^2\}$.
\end{proof}

\bibliographystyle{plain}
\bibliography{bibfile}
\end{document}